\documentclass[11pt,dvips,twoside,letterpaper]{article}
\RequirePackage[latin1]{inputenc} \RequirePackage[T1]{fontenc}
\usepackage{pslatex}
\usepackage{fancyhdr}
\usepackage{graphicx}
\usepackage{geometry}

\def\figurename{Figure} 
\makeatletter
\renewcommand{\fnum@figure}[1]{\figurename~\thefigure.}
\makeatother

\def\tablename{Table} 
\makeatletter
\renewcommand{\fnum@table}[1]{\tablename~\thetable.}
\makeatother

\usepackage{amsmath}
\usepackage{amssymb}
\usepackage{amsfonts}
\usepackage{amsthm,amscd}

\newtheorem{theorem}{Theorem}[section]
\newtheorem{lemma}[theorem]{Lemma}

\theoremstyle{definition}

\theoremstyle{remark}
\newtheorem{remark}[theorem]{Remark}

\numberwithin{equation}{section}

\def\P{\mathbb P}
\def\R{\mathbb R}
\def\E{\mathbb E}

\def\M{\mathbb M}

\def\E{\mathbb E}

\def\cal{\mathcal}

\begin{document}
\title{\bfseries\scshape{Numerical scheme for backward doubly stochastic differential equations}}
\author{\bfseries\scshape Auguste Aman\thanks{augusteaman5@yahoo.fr;\; Supported by AUF post doctoral grant 07-08, Réf:PC-420/2460}\thanks{This work was partially performed when the author visit Université Cadi Ayyad of Marrakech}\\
UFR Mathématiques et Informatique\\ Universit\'{e} de Cocody, \\BP 582 Abidjan 22, C\^{o}te d'Ivoire}

\date{}
\maketitle \thispagestyle{empty} \setcounter{page}{1}

\begin{abstract}
We study a discrete-time approximation for solutions of systems of decoupled forward-backward doubly
stochastic differential equations (FBDSDEs). Assuming that the coefficients are Lipschitz-continuous, we prove the convergence of the scheme when the step of time discretization, $|\pi|$ goes to zero. The rate of convergence is exactly equal to $|\pi|^{1/2}$. The proof is based on a generalization of a remarkable result
on the $^{2}$-regularity of the solution of the backward equation derived by J. Zhang \cite{Z}.
\end{abstract}


\vspace{.08in} \noindent \textbf{Keywords}: Backward doubly SDEs; Discrete-time approximation.

\section{Introduction}
\setcounter{theorem}{0} \setcounter{equation}{0}
In this paper, we study a discrete time approximation scheme for the solution of a system of
the (decoupled) forward-backward doubly stochastic differential equations (FBDSDEs, in short) on the time interval $[0,T]$:
\begin{eqnarray}
\left\{
\begin{array}{l}
X_{t}=x+\int_{0}^{t}b(X_{s})ds+\int_{0}^{t}\sigma(X_{s})dW_{s}\\\\
Y_{t}=h(X_{T})+\int_{t}^{T}f(s,X_{s},Y_{s},Z_{s})\,
ds+\int_{t}^{T}g(s,X_{s},Y_{s})d\overleftarrow{B_{s}} -\int_{t}^{T}Z_{s}dW_{s}.\label{FBSDE}
\end{array}\right.
\end{eqnarray}
Here $W$ and $B$ are two independent  Brownian motion such that, the integral with respect to ${B_t}$ is a backward It\^{o} integral
and the one with respect to ${ W_t}$ is a standard forward It\^{o} integral. Let us note that such equations naturally appear in probabilistic interpretation of stochastic partial differential equations (SPDEs, in short). Indeed, under standard Lipschitz assumptions on the coefficients $b, \sigma, f, g$, and $h$, the existence and uniqueness of the solution $(Y,Z)$ have been proved by Pardoux and Peng \cite{PP1}. Moreover, they give the link between the classical solution of SPDE in the following. More precisely let consider the SPDE
\begin{eqnarray}
-\frac{\partial}{\partial t}u(t,x)-[{
L}u(t,x)-f(t,x,u(t,x),\sigma^{*}(x)\nabla u(t,x))]-g(t,x,u(t,x))\lozenge B_{s}=0,\nonumber\\
u(T,x)=h(x),\label{SPDE}
\end{eqnarray}
where $\lozenge$ denotes the Wick product and, thus, indicates that the differential is to understand in It\^{o}'s sense, and
\begin{eqnarray*}
L=\frac{1}{2}\sum_{i,j}^{n}(\sigma\sigma^{*})_{ij}(x)\frac{\partial^2}{\partial x_{i}\partial x_j}+\sum_{i}^{n}b_{i}(x)\frac{\partial}{\partial x_{i}}.
\end{eqnarray*}
Under more strengthen assumptions (the coefficients $f,\, g$ and $h$ are $\mathcal{C}^{3}$ class), the component $Y$ of the solution of $(\ref{FBSDE})$ is related to the classical solution $u$ of SPDE (\ref{SPDE}), in the sense that
\begin{eqnarray}
Y_t=u(t,X_t).\label{FBSDE-SPDE}
\end{eqnarray}
Furthermore, Buckdahn and Ma relax the assumptions of coefficient to standard Lipschitz one and they proved among other that  the relation $(\ref{FBSDE-SPDE})$ give the stochastic viscosity solution of SPDE $(\ref{SPDE})$. Thus, solving $(\ref{FBSDE})$ or $(\ref{SPDE})$ is essentially the same. However it is known that only a limited number of BDSDE can be solved explicitly. In order to solved the large class of BDSDE and of course provide an alternative to classical numerical schemes for a large class of SPDE, the numerical method and numerical algorithm is very helpful.

In the one stochastic case, i.e $g\equiv0$, the numerical approximation of $(\ref{FBSDE})$ has already been
studied in the literature; see e.g. Zhang \cite{Z}, Bally and Pages \cite{BP}, Bouchard and Touzi \cite{Tal} or
Gobet et al. \cite{Gal}. In \cite{Tal}, the authors suggest the following implicit scheme. Given a partition regular grid
$\pi: 0=t_0<t_1<....<t_n=T$ of the interval $[0,T]$, they approximate $X$ by its well-know Euler scheme $X^{\pi}$ and $(Y, Z)$, by the
discrete-time process $(Y^{\pi}_{t_i}, Z^{\pi}_{t_i})_{0\leq i\leq n}$ defined backward by
\begin{eqnarray*}
\left\{
\begin{array}{l}
Z^{\pi}_{t_i}=\frac{1}{\Delta_{i+1}^{\pi}}\E\left[Y^{\pi}_{t_{i+1}}\Delta^{\pi} W_{t_{i+1}}|\mathcal{F}_{t_i}\right]\\\\
Y^{\pi}_{t_i}=\E\left[Y^{\pi}_{t_{i+1}}|\mathcal{F}_{t_i}\right]+\Delta_{i+1}^{\pi}f(t_i,X^{\pi}_{t_i},Y^{\pi}_{t_i},Z^{\pi}_{t_i}),
\end{array}\right.
\end{eqnarray*}
where $\displaystyle{Y^{\pi}_{t_n}=h(X^{\pi}_T)},\ \Delta^{\pi} W_{i+1}=W_{t_{i+1}}-W_{t_i}$ and $\Delta_{i+1}^{\pi}=t_{i+1}-t_i$. Then, it turn out that
the discretization error
\begin{eqnarray*}
Err_{\pi}(Y,Z)=\left\{\sup_{0\leq t\leq T}\E|Y_{t}-Y_{t}^{\pi}|^{2}+\int_{0}^{T}\E\left[|Z_{s}-Z^{\pi}_{s}|^{2}\right]ds\right\}^{1/2}
\end{eqnarray*}
is intimately related to the quantity
\begin{eqnarray*}
\sum_{i=1}^{n-1}\int^{t_{i+1}}_{t_i}\E\left[|Z_{s}-\tilde{Z}_{t_i}|^{2}\right]ds\;\;\;\;\;\mbox{where}\;\;\;\;\;\tilde{Z}_{t_i}=\frac{1}{\Delta_{i+1}^{\pi}}
\E\left[\int_{t_i}^{t_{i+1}}Z_s|\mathcal{F}_{t_i}\right].
\end{eqnarray*}
Under Lipschitz continuity conditions on the coefficients, Zhang \cite{Z} was able to prove that the
latter is of order of $|\pi|$, the partition's mesh. This remarkable result allows them to derive the bound $Err_{\pi}(Y, Z)\leq C|\pi|^{1/2}$. Observe that this rate of convergence cannot be improved in general. Consider, for example, the case where $X$ is equal to the Brownian motion $W$, $h$ is the identity, and $f = 0$.
Then, $Y = W$ and $Y^{\pi}_{t_i} = W_{t_i}$.

In this paper, we extend the approach of Bouchard and Touzi \cite{Tal}, and approximate the solution of $(\ref{FBSDE})$ by the following backward scheme.
\begin{eqnarray*}
\left\{
\begin{array}{l}
Z^{\pi}_{t_{i}}=\frac{1}{\Delta_{i+1}}\E^{\pi}_{i}\left[\left(Y_{t_{i+1}}^{\pi}
+g(t_{i+1},X_{t_{i+1}}^{\pi},Y_{t_{i+1}}^{\pi})\Delta B_{i+1}\right)\Delta W_{i+1}\right],\label{a5}\\\\
Y^{\pi}_{t_{i}}=\E_{i}^{\pi}\left[Y_{t_{i+1}}^{\pi}+g(t_{i+1},X^{\pi}_{t_{i+1}},Y^{\pi}_{t_{i+1}})\Delta B_{i+1}\right]
+f(t_{i},X^{\pi}_{t_{i}},Y^{\pi}_{t_{i}},Z^{\pi}_{t_{i}})\Delta_{i+1}
\end{array}\right.
\end{eqnarray*}
where $Y_{t_n}^{\pi}= h(X^{\pi}_{T})$ and $\Delta B_{i+1}=B_{t_{i+1}}-B_{t_i}$. By adapting the arguments of Bouchard and Touzi \cite{Tal}, we first prove that our
discretization error $Err_{\pi}(Y, Z)$ converge to $0$ as the step of the discretization $|\pi|$ tends to $0$. We then provide upper bounds on
\begin{eqnarray*}
\max_{i<n}\sup_{0\leq t\leq t_i}\E|Y_{t}-Y_{t_i}|^{2}+\sum_{i}^{n-1}\int^{t_{i+1}}_{t_i}\E\left[|Z_{s}-\tilde{Z}_{t_i}|^{2}\right]ds.
\end{eqnarray*}
When the coefficients are Lipschitz continuous, we obtain
\begin{eqnarray*}
\max_{i<n}\sup_{0\leq t\leq t_i}\E|Y_{t}-Y_{t_i}|^{2}+\sum_{i}^{n-1}\int^{t_{i+1}}_{t_i}\E\left[|Z_{s}-\tilde{Z}_{t_i}|^{2}\right]ds<C|\pi|.
\end{eqnarray*}
This extends to our framework the remarkable result derived by Zhang \cite{Z}. It allows us to show that our discrete-time scheme
achieves, under the standard Lipschitz conditions, a rate of convergence exactly equal to $|\pi|^{1/2}$.

Observe that, in opposition to algorithms based on the approximation of the Brownian motion by discrete processes taking a finite number of possible values (see e.g. \cite{Y} and the references therein), our scheme does not provide a fully implementable numerical procedure, since it involves the computation of a large number of conditional expectations.

This paper is organized as follows. In Section 2, we introduce some fundamental knowledge and assumptions of BDSDEs and give extension of the remarkable $L^2$-regularity results derived by Zhang \cite{Z} to the doubly stochastic case, which is our first main result. In Section 3, we describe the approximation
scheme and state convergence result, our second main result.

{\bf Notations.}  We shall denote by $\M^{n,d}$ the set of all $n\times d$ matrices with real coefficients. We simply denote $\R^{n}= \M^{n,1}$ and
$\M^{n}= \M^{n,n}$. We shall denote by $\|a\|=(\sum_{i,j} a^{2}_{i,j})^{1/2}$ the Euclidian norm on $\M^{n,d}$, $a^{*}$ the transpose of $a$, $a^{k}$ the $k$-th column of $a$. To simplify, we denote respectively by $|x|$ and $a_k$, the norm and the  the $k$-th component of $a\in\R^{n}$. Finally, we denote by  $x.y=\sum_{i}x_iy_i$ the scalar product in $\R^{n}$.
\section{Forward-Backward doubly SDEs}
\subsection{Preliminaries and Assumptions}
\setcounter{theorem}{0} \setcounter{equation}{0}
Let $(\Omega_{1},\mathcal{F}_{1},\mbox{I\hspace{-.15em}P}_{1})$ and
$(\Omega_{2},\mathcal{F}_{2},\mbox{I\hspace{-.15em}P}_{2})$ be two complete probability  spaces and $T>0$ a fixed final time. Throughout this paper we consider
$\{W_{t}, 0\leq t\leq T\}$ and $\{B_{t}, 0\leq t\leq T\}$ two mutually independent standard Brownian motions processes, with values respectively in $\R^{d}$ and $\R^{\ell}$, defined respectively on $(\Omega_{1},\mathcal{F}_{1},\mbox{I\hspace{-.15em}P}_{1})$ and
$(\Omega_{2},\mathcal{F}_{2},\mbox{I\hspace{-.15em}P}_{2})$. For any process $\displaystyle{\left(\eta_{s}: 0\leq
s\leq T\right)}$ defined on
$(\Omega_{i},\mathcal{F}_{i},\mbox{I\hspace{-.15em}P}_{i}),\,(
i=1,\,2)$, we denote
\begin{eqnarray*}
\mathcal{F}%
^{\eta}_{s,t}=\sigma\{\eta_{r}-\eta_{s}, s\leq r \leq t\}\vee
\mathcal{N}, \,\,\mathcal{F}^{\eta}_{t}=\mathcal{F}^{\eta}_{0,t}.
\end{eqnarray*}
 In the sequel of the paper unless otherwise specified we denote
\begin{eqnarray*}
\Omega=\Omega_{1}\times\Omega_{2},\,\,\mathcal{F}=\mathcal{F}_{1}\otimes\mathcal{F}_{2}\,\,
\mbox{ and}\,\, \P=\P_{1}\otimes\P_{2}.
\end{eqnarray*}
Moreover, we put
\begin{eqnarray*}
\mathcal{F}_{t}=\mathcal{F}^{W}_{t}\otimes\mathcal{F}_{T}^{B}\vee\mathcal{N}
\end{eqnarray*}
where $\mathcal{N}$ is the collection of $\P$-null sets and denote ${\bf F}=(\mathcal{F}_{t})_{t\geq 0}$. Further, for random
variables $\epsilon(\omega_{1}),\,\omega_{1}\in \Omega_{1}$ and $\beta(\omega_{2}),\, \omega_{2}\in\Omega_{2}$, we view them as random variables in $\Omega $ by the following identification:
\begin{eqnarray*}
\epsilon(\omega)=\epsilon(\omega_{1});\,\,\,\,\,\beta(\omega)=\beta(\omega_{2}),\; \omega=(\omega_1,\omega_2).
\end{eqnarray*}
Given $C>0$, we consider two functions $b:\R^{d}\rightarrow\R^{d}$ and
$\sigma:\R^{d}\rightarrow\M^{d}$ two
functions satisfying the Lipschitz condition
\begin{description}
\item $({\bf H1})\, |b(x)-b(x')|+\|\sigma(x)-\sigma(x')\|\leq C|x-x'|, \,\, \forall\,\, x,x'\in\R^{d}$.
\end{description}
Then it is well-known that (see e.g Karatzas and Shreve \cite{KS}), for any initial condition $x\in\R^{d}$, the forward stochastic differential equation
\begin{eqnarray}
 X_{t}=x+\int_{0}^{t}b(X_{s})ds+\int_{0}^{t}\sigma(X_{s})dW_{s},\quad t\in[0,T]
 \label{a_{0}}
\end{eqnarray}
has a $\mathcal{F}_{t}$-adapted solution $(X_{t})_{0\leq t\leq T}$ satisfying
\begin{eqnarray*}
\E(\sup_{0\leq t\leq T}|X_{t}|^{2})<\infty.
\end{eqnarray*}

Before introducing the backward doubly SDE, we need to define some additional notations. Given some real number $p\geq 2$, we denote by ${\cal{S}}^{p}$
the set of real valued adapted c\`{a}dl\`{a}g processes $Y$ such that
\begin{eqnarray*}
\|Y \|_{{\cal{S}}^{p}}=\E\left[\sup_{0\leq t\leq T}|Y_t|^{p}\right]<\infty.
\end{eqnarray*}
${\cal{H}}^{p}$ is the set of progressively measurable $\R^{d}$-valued processes $Z$ such that
\begin{eqnarray*}
\|Z\|_{{\cal{H}}^{p}}=\E\left[\int_{0}^{T}|Z_t|^{p}dt\right]^{1/p}<\infty.
\end{eqnarray*}
The set ${\cal{B}}^{p}={\cal{S}}^{p}\times{\cal{H}}^{p}$ is endowed with the norm
\begin{eqnarray*}
\|(Y,Z)\|_{{\cal{B}}^{p}}=\left(\|Y \|^{p}_{{\cal{S}}^{p}}+\|Z \|^{p}_{{\cal{H}}^{p}}\right)^{1/p}.
\end{eqnarray*}

The aim of this paper is to study a discrete-time approximation of the pair $(Y, Z)$ solution
on $[0, T ]$ of the backward doubly stochastic differential equation
\begin{eqnarray}
Y_{t}=h(X_{T})+\int_{t}^{T}f(s,X_{s},Y_{s},Z_{s})\,
ds+\int_{t}^{T}g(s,X_{s},Y_{s})d\overleftarrow{B_{s}} -\int_{t}^{T}Z_{s}dW_{s},\;\;
0\leq t\leq T. \label{a00}
\end{eqnarray}
By a solution, we mean a triplet $(Y, Z)\in {\cal{B}}^{p}$ satisfying $(\ref{a00})$.

In order to ensure the existence and uniqueness of a solution to $(\ref{a00})$, and the convergence of our discrete-time approximation, we assume that the map
$f:\, [0,T]\times\R^{d}\times\R\times\R^{d}\xrightarrow{}\R,\,
g:\, [0,T]\times\R^{d}\times\R\xrightarrow{}\R^{\ell }$ and $h:\,
\R^{d}\xrightarrow{}\R$ satisfied the Lipschitz condition:
\begin{description}
\item $({\bf H2})$
\begin{description}
\item $(i)\; |f(s,x,y,z)-f(s',x',y',z')|^{2}\leq C\left(|s-s'|^{2}+|x-x'|^{2}+|y-y'|^{2}+|z-z'|^{2}\right)$
\item $(ii)\; |g(s,x,y)-g(s,x',y')|^{2} \leq  C(|s-s'|^{2}+|x-x'|^2+|y-y'|^{2})$
\item $(iii)\; |h(x)-h(x')|^{2}\leq C|x-x'|^{2}$
\end{description}
\end{description}
for some constant $C>0$ independent of all the variables.
\begin{remark}
In order to ensure the existence and uniqueness to the solution of  $(\ref{a00})$, we need only that $f$ and $g$ are Lipschitz with respect variables $y$ and $z$. See Pardoux and Peng \cite{PP1} for more detail.
\end{remark}

The following lemmas collect without proof, some standard results in SDE and BDSDE literature. We list them for ready references. For ease of notation, we shall denote by $C_p$ a generic constant depending only on $p$, the constants $C,\, b(0),\, \sigma(0),\, h(0)$ and $T$ and the functions $f(.,0,0,0)$ and $g(.,0,0)$.
\begin{lemma}
\label{L1}
Assume $b$ and $\sigma$ satisfy $({\bf H1})$ and $X$ be the unique solution of forward SDE $(\ref{a_{0}})$. Then
\begin{eqnarray*}
\|X\|_{{\cal{S}}^{p}}^{p}\leq C_p(1+|x|^{p})
\end{eqnarray*}
and
\begin{eqnarray*}
\E\left[|X_t-X_s|^{p}\right]\leq C_p(1+|x|^{p})|t-s|^{p/2}.
\end{eqnarray*}
\end{lemma}

\begin{lemma}
\label{L2}
Assume $({\bf H2})$ and $(Y,Z)$ be the unique solution of backward doubly SDE $(\ref{a00})$. Then
\begin{eqnarray*}
\|(Y,Z)\|_{{\cal{B}}^{p}}^{p}\leq C_{p}(1+|x|^{p})
\end{eqnarray*}
and
\begin{eqnarray*}
\E\left[|Y_t-Y_s|^{p}\right]\leq C_{p}\left\{(1+|x|^{p})|t-s|^{p-1}+\|Z\|_{{\cal{H}}^{p}}^{p}\right\}.
\end{eqnarray*}
\end{lemma}

\subsection{$L^{2}$-regularity}
In this subsection we establish the first main result of this
paper, which we shall call the $L^2$-regularity. Such a
regularity, plays a key role
for deriving the rate of convergence of our numerical scheme in Section 4 and, in our mind generalized Theorem 3.4.3 in \cite{Z}.

To begin with, let  $\pi:0=t_0<...<t_n=T$ be a partition of the time interval $[0,T]$, with $|\pi|=\max_{1\leq i\leq n}|t_{i-1}-t_{i}|$, the size of the partition.
and $X$ be the solution of the forward SDE $(\ref{a_{0}})$. We denote by $(Y,Z)$ the solution of the following backward SDE
\begin{eqnarray}
Y_{t}=\phi^{\pi}(X_{t_0},...,X_{t_n})+\int_{t}^{T}f(s,X_s,Y_s,Z_s)\,
ds+\int_{t}^{T}g(s,X_s,Y_s)d\overleftarrow{B_{s}} -\int_{t}^{T}Z_{s}dW_{s},\label{eqpart0}
\end{eqnarray}
the generalized form of BDSDE $(\ref{a00})$. Next, for $X^{\pi}$ the well-know Euler scheme of $X$ that will be explicit in Section 3, let $(Y^{\pi},Z^{\pi})$ be the adapted solution to the following BDSDE
\begin{eqnarray}
Y_{t}^{\pi}=\phi^{\pi}(X_{t_0}^{\pi},...,X_{t_n}^{\pi})+\int_{t}^{T}f(s,X^{\pi}_s,Y^{\pi}_s,Z^{\pi}_s)\,
ds+\int_{t}^{T}g(s,X^{\pi}_s,Y^{\pi}_s)d\overleftarrow{B_{s}} -\int_{t}^{T}Z_{s}^{\pi}dW_{s}.\label{eqpart}
\end{eqnarray}

To simplify presentations, in what follows we assume that $X_t,\ X^{\pi}_t\in\R^d$, and the other processes are all one-dimensional. But the results can be extended to cases with higher-dimensional on this processes without significant difficulties. For simplicity we also denote by $\Xi=(X,Y),\,\Theta=(X,Y,Z)$ and $\Xi^{\pi}=(X^{\pi},Y^{\pi}),\,\Theta^{\pi}=(X^{\pi},Y^{\pi},Z^{\pi})$.

Now we have
\begin{lemma}
\label{L5}
Assume the functions $\phi^{\pi}:\R^{d(n+1)}\rightarrow\R, \, f:[0,T]\times\R^{d}\times\R^{2}\rightarrow\R$ and $g:[0,T]\times\R^{d}\times\R\rightarrow\R$ satisfying assumptions $({\bf H2})$ with adequate norm. For each $1\leq i\leq n$, we define
\begin{eqnarray*}
\tilde{Z}^{\pi}_{t_{i-1}}&=&\frac{1}{\bigtriangleup_{i}^{\pi}}\E_{i-1}^{\pi}\left[\int^{t_i}_{t_{i-1}}Z_{s}ds\right],
\label{a'6}
\end{eqnarray*}
where $\E_{i-1}^{\pi}(.)=\E(.|\mathcal{F}_{t_{i-1}}^{W}\vee \mathcal{F}^{B}_{T})$.
Then
\begin{eqnarray}
&&\limsup_{\pi\rightarrow 0}|\pi|^{-1}\E\left[\max_{1\leq i\leq n}\sup_{t_{i-1}\leq t\leq t_i}|Y_{t}-Y_{t_{i-1}}|^{2}+\sum_{i=1}^{n}\int^{t_i}_{t_{i-1}}|Z_{s}-\tilde{Z}^{\pi}_{t_{i-1}}|^{2}ds\right]< \infty.
\label{a"61}
\end{eqnarray}
\end{lemma}
Before prove this important theorem, we state the following needed result. To this end let us assume the following: $\phi^{\pi}\in C_{b}^{1}(\R^{d(n+1)})$, $f\in C_{b}^{0,1}([0,T]\times\R^{d}\times\R^{2})$ and $g\in C_{b}^{0,1}([0,T]\times\R^{d}\times\R)$. Moreover, for all $x = (x_0,....,x_n)\in\R^{d(n+1)}$,
\begin{eqnarray}
\sum_{i=0}^{n}|h_{x_i}^{\pi}(x)|\leq C.\label{bound}
\end{eqnarray}
We also design by $\varphi_{u}$ the partial differential of $\varphi$ which respect the variable $u$.

Next, we denote by $\nabla X^{\pi}$ the solution of the following variational equation:,
\begin{eqnarray}
\nabla X_t^{\pi}=I_d+\int^{t}_{0}b_x(X^{\pi}_r)\nabla X^{\pi}_rdr+\int^{t}_{0}\sigma_{x}(X^{\pi}_r)\nabla X^{\pi}_rdW_r,
\end{eqnarray}
and by $(\nabla^{i}Y^{\pi},\nabla^{i}Z^{\pi})$ the solution of the following BDSDE on $[t_{i-1}, T ]$:
\begin{eqnarray}
\nabla^{i}Y_{t}^{\pi}&=&\sum_{j\geq i}^{n}h^{\pi}_{x_j}(X_{t_0}^{\pi},...,X_{t_n}^{\pi})\nabla X^{\pi}_{t_j}+\int^{T}_{t}[f_{x}(\Theta^{\pi}_r)\nabla
X_{r}^{\pi} +f_{y}(\Theta^{\pi}_r)\nabla^{i}Y_{r}^{\pi}
+f_{z}(\Theta^{\pi}_r)\nabla^{i}Z_{r}^{\pi}]dr\nonumber\\
&&+\int^{T}_{t}[g_x(\Xi^{\pi}_r)\nabla X_{r}^{\pi}+g_y(\Xi^{\pi}_r)\nabla^{i} Y_{r}^{\pi})]\overleftarrow{dB}_{r}-\int^{T}_{t}\nabla^{i}
Z_{r}^{\pi}dW_{r}, \;\; t\in[t_{i+1},T],\nonumber\\ \mbox{for}\, i = 1,...,n.
\label{eqvaria}
\end{eqnarray}
On the other hand, we denote by
\begin{align}
\nabla^{\pi}Y_{t}^{\pi}=\sum_{i=1}^{n}\nabla^{i}Y_{t}^{\pi}{\bf 1}_{[t_{i-1},t_{i})}(t)+\nabla^{n}Y_{T^-}^{\pi}{\bf 1}_{\{T\}}(t),\;\;\;\; t\in[0,T];
\end{align}
hence $\nabla^{\pi}Y^{\pi}$ is a c\`{a}dl\`{a}g process.

For application convenience, we shall rewrite $\nabla^{\pi}Y^{\pi}$ in another form. Note that for
each $i$ $(\ref{eqvaria})$ is linear. Let $(\gamma^{0},\zeta^{0})$ and $(\gamma^{j},\zeta^{j}), \; j = 1,..., n$ be the adapted solutions
of the BDSDEs
\begin{eqnarray}
\gamma^{0}_{t}&=&\int^{T}_{t}[f_{x}(\Theta^{\pi}_r)\nabla X_{r}^{\pi} +f_{y}(\Theta^{\pi}_r)\gamma_{r}^{0}+f_{z}(\Theta^{\pi}_r)\zeta^{0}_{r}]dr
\nonumber\\
&&+\int^{T}_{t}[g_x(\Xi^{\pi}_r)\nabla X_{r}^{\pi}+g_y(\Xi^{\pi}_r)\gamma^{0}_{r}]d\overleftarrow{B}_{r}-\int^{T}_{t}\zeta^{0}_{r}dW_{r},\label{c8}\\
\gamma_{t}^{j}&=& h^{\pi}_{x_j}(X^{\pi}_{t_0},.....,X^{\pi}_{t_n})\nabla X^{\pi}_{t_j}+\int^{T}_{t}[f_{y}(\Theta^{\pi}_r)\gamma_{r}^{j}+f_{z}(\Theta^{\pi}_r)\zeta_{r}^{j}]dr\nonumber\\
&&+\int^{T}_{t} g_y(\Xi^{\pi}_r)\gamma_{r}^{j}d\overleftarrow{B}_{r}-\int^{T}_{t}\zeta_{r}^{j}dW_{r},\nonumber
\end{eqnarray}
respectively, then we have the following decomposition:
\begin{eqnarray}
\nabla^{i}Y_{s}=\gamma_{s}^{0}+\sum_{j\geq i}\gamma_{s}^{j},\;\;\;\;\;\;s\in [t_{i-1},t_i).\label{c9}
\end{eqnarray}

We may simplify $(\ref{c9})$ further. Let us define, for any $\eta\in L^{1}({\bf F},[0,T])$ and $(\Theta_1,\Theta_2)\in L^{2}({\bf F},[0,T];\R)\times L^{2}({\bf F},[0,T];\R)$,
\begin{eqnarray*}
\Lambda_{t}^{s}(\eta)&=&\exp\left(\int_{s}^{t}\eta(r)dr\right),\;\;\; s,t\in[0,T],\\
^1{\mathcal{E}}_t^{s}(\Theta_1)&=&\exp\left\{\int_{s}^{t}\Theta_1(r) dW_r-\frac{1}{2}\int_s^t |\Theta_1(r)|^{2}dr\right\},\;\;\; s,t\in[0,T],\\
^2\mathcal{E}_t^{s}(\Theta_2)&=&\exp\left\{\int_{s}^{t}\Theta_2(r) d\overleftarrow{B}_r-\frac{1}{2}\int_s^t |\Theta_2(r)|^{2}dr\right\},\;\;\; s,t\in[0,T]. \label{m}
\end{eqnarray*}
($^1\mathcal{E}_t^{s}(\Theta_1)$ and $^2\mathcal{E}_t^{s}(\Theta_2)$ are respectively the well known Dal\'{e}an-Dade stochastic exponential of $\Theta_1$ with respect $W$ and $\Theta_2$ with respect $B$). Then it is easily checked that, for any $p > 0$, one has
\begin{eqnarray}
[^i\mathcal{E}_t^{s}(\Theta_i)]^{p}=\,^i\mathcal{E}_t^{s}(p\Theta_i)\Lambda_{t}^{s}(\frac{p(p-1)}{2}|\Theta_i|^2),\label{lambda}
\end{eqnarray}
and
\begin{eqnarray}
[^i\mathcal{E}_t^{s}(\Theta_i)]^{-1}= \,^i\mathcal{E}_t^{s}(-\Theta_i)\Lambda_{t}^{s}(|\Theta_i|^2),\;\;\; i=1,2.\label{lambda1}
\end{eqnarray}
In particular, we denote, for $s, t\in[0, T]$,
\begin{eqnarray}
\Lambda_{t}^{s}&=&\Lambda_{t}^{s}(-f_y)\ ^2\mathcal{E}_t^{s}(-g_y),\;\;\, M_t^{s}=\ ^1\mathcal{E}_t^{s}(f_z),\label{lambda2}
\end{eqnarray}
and if there is no danger of confusion, we denote $\Lambda_{.}= \Lambda_{.}^{0}$ and $M_.=M_.^{0}$.
Since $f_z$ is uniformly bounded, by Girsanov's Theorem (see, e.g., \cite{KS}) we know that $M$ is a $\P$-martingale on $[0,T]$, and $\widetilde{W}_{t}=W_{t}-\int_{0}^{t}f_z (\Theta^{\pi}_s)dr,\, t\in[0,T]$ is an ${\bf F}$-Brownian motion on the new probability space $(\Omega,\mathcal{F},\widetilde{\P}),$ where $\widetilde{\P}$ is defined by $\frac{d\widetilde{\P}}{d\P}=M_T$. Moreover noting that $f_y,\, f_z$ and $g_y$ are uniformly bounded, by virtue of $(\ref{lambda})$ and $(\ref{lambda1})$ one can
deduce easily from $(\ref{lambda2})$ that, for $p\geq 1$, there exists a constant $C_p$ depending only on $T, C$ and $p$, such that
\begin{eqnarray}
&&\E\left(\sup_{0\leq t\leq T}|\Lambda_{t}|^{p}+|\Lambda_{t}^{-1}|^{p}\right)\leq C_{p};\;\;\E\left(\sup_{0\leq t\leq T}[|M_{t}|^{p}+|M_{t}^{-1}|^{p}]\right)\leq C_{p};\nonumber\\
\nonumber\\
&&\E\left(|\Lambda_{t}-\Lambda_s|^{p}+|\Lambda_{t}^{-1}-\Lambda^{-1}_s|^{p}\right)\leq C_{p}|t-s|^{p/2};\label{SE}\\
\nonumber\\
&&\E\left(|M_{t}-M_s|^{p}+|M_{t}^{-1}-M^{-1}_s|^{p}\right)\leq C_{p}|t-s|^{p/2}.
\nonumber
\end{eqnarray}
\begin{lemma}
\label{L3}
Assume $\sigma,\,b\in C_{b}^{1}$ and $f,\,g,\,l$ satisfy the previous assumptions. Then for all $i=1,...,n$
\begin{eqnarray*}
\nabla^i Y_{t}^{\pi}=\left(\xi_{t}^{0}+\sum_{j\geq i}\xi^{j}_{t}\right)M^{-1}_t\Lambda_{t}-\int_{0}^{t}f_{x}(\Theta^{\pi}_r)\nabla X_{r}^{\pi}\Lambda^{-1}_{r}dr\Lambda_{t}-\int_{0}^{t}g_x(\Xi^{\pi}_r)\nabla X_{r}^{\pi}\Lambda^{-1}_{r}d\overleftarrow{B}_r\Lambda_{t},
\end{eqnarray*}
where $\xi_{t}^{0}$ and $\xi_{t}^{j}$, for $j=1,\cdot\cdot\cdot,n$ will be explicit in the proof.
\end{lemma}
\begin{proof}
Let us denote the following:
\begin{eqnarray*}
\begin{array}{l}
\widetilde{\xi}^{0}=\int_{0}^{T}f_{x}(\Theta^{\pi}_r)\nabla X_{r}^{\pi}\Lambda^{-1}_{r}dr+\int_{0}^{T}g_x(\Xi^{\pi}_r)\nabla X_{r}^{\pi}\Lambda^{-1}_{r}d\overleftarrow{B}_{r},\;\;\widetilde{\zeta}^{0}_t=\zeta^{0}\Lambda^{-1}_t,\\\\ \widetilde{\gamma}^{0}_t=\gamma^{0}_t\Lambda^{-1}_t+\int_{0}^{t}f_{x}(\Theta^{\pi}_r)\nabla X_{r}^{\pi}\Lambda^{-1}_{r}dr+\int_{0}^{t}g_x(\Xi^{\pi}_r)\nabla X_{r}^{\pi}\Lambda^{-1}_{r}d\overleftarrow{B}_{r}\\\\
\widetilde{\xi}^{i}=h^{\pi}(X^{\pi}_{t_0},\cdot\cdot\cdot,X^{\pi}_{t_n})\nabla X_{t_i}^{\pi}\Lambda^{-1}_{T},\;\; \widetilde{\zeta}^{i}_t=\zeta^{i}_t\Lambda^{-1}_t,\;\;\widetilde{\gamma}^{i}_t=\gamma^{i}_t\Lambda^{-1}_t.
\end{array}
\end{eqnarray*}
Then, using integration by parts and equation $(\ref{c8})$ we have, for $i=0,1,....,n$,
\begin{eqnarray*}
\widetilde{\gamma}^{i}_{t}=\widetilde{\xi}^{i}
-\int_{t}^{T}\widetilde{\zeta}^{i}_rd\widetilde{W}_{r}, \;\; t\in[0,T],
\end{eqnarray*}
so that, $\int_{0}^{t}\widetilde{\zeta}^{i}_r d\widetilde{W}_{r}$ being a uniformly integrable martingale with in particular zero expectation, we get
\begin{eqnarray*}
\widetilde{\gamma}^{i}_{t}=\widetilde{\E}(\widetilde{\xi}^{i}|\mathcal{F}_t).
\end{eqnarray*}
Therefore, by the Bayes rule (see e.g, \cite{KS} Lemma 3.5.3) we have for $t\in[0,T]$
\begin{eqnarray*}
\gamma^{0}_{t}&=&\widetilde{\gamma}^{0}_{t}\Lambda_t-\int_{0}^{t}f_{x}(\Theta^{\pi}_r)\nabla X_{r}^{\pi}\Lambda^{-1}_{r}dr\Lambda_{t}-\int_{0}^{t}g_{x}(\Xi^{\pi}_r)\nabla X_{r}^{\pi}\Lambda^{-1}_{r}d\overleftarrow{B}_{r}\Lambda_{t}\nonumber\\
&=&\xi_{t}^{0} M^{-1}_t\Lambda_t-\int_{0}^{t}f_{x}(\Theta^{\pi}_r)\nabla X_{r}^{\pi}\Lambda^{-1}_{r}dr\Lambda_{t}-\int_{0}^{t}g_x(\Xi^{\pi}_r)\nabla X_{r}^{\pi}\Lambda^{-1}_{r}d\overleftarrow{B}_{r}\Lambda_{t},\nonumber\\
\gamma^{i}_{t}&=&\widetilde{\gamma}^{i}_{t}\Lambda_t=\widetilde{\E}\left(\widetilde{\xi}^{i}|\mathcal{F}_{t}\right)\Lambda_t
=\E\left(M_T\widetilde{\xi}^{i}|\mathcal{F}_{t}\right)M^{-1}_t\Lambda_t=\xi_{t}^{i} M^{-1}_t\Lambda_t,
\end{eqnarray*}
where, for $i=0,1,....,n$,
\begin{eqnarray}
\xi_{t}^{i}=\E\left(M_T\widetilde{\xi}^{i}|\mathcal{F}_{t}\right)=\E\left(M_T\widetilde{\xi}^{i}\right)+\int_{0}^{t}\chi_s^{i} dW_s.\label{xi}
\end{eqnarray}
Note that the boundedness of $f_z$ and $(\ref{SE})$ imply that $M_{T}\in L^{p}(\Omega)$ and $\nabla X\in L^{p}({\bf F},C([0,T];\M^{d}))$ for all $p\geq 2$. Therefore for each $p\geq 1$, $(\ref{bound})$ leads to
\begin{eqnarray*}
\E\left\{\sum_{j=0}^{n}|M_T\widetilde{\xi}^{j}|\right\}\leq C\E\left\{|M_{T}|^{p}\sup_{0\leq t\leq T}|\nabla X_t|^{p}\right\}\leq C.
\end{eqnarray*}
In particular, for each $j=0,\cdot\cdot\cdot,n$, $M_{T}\widetilde{\xi}^{j}\in L(\mathcal{F}_{T})$. So  $(\ref{xi})$ makes sens. Finally the result follows by $(\ref{c9})$.
\end{proof}
\begin{proof} [Proof of Lemma 2.4]
For all $1\leq i\leq n$ and each $t\in[t_{i-1},t_i)$, applying Lemma\,$\ref{L2}$, we get
\begin{eqnarray*}
\E\left(|Y_{t}-Y_{t_{i-1}}|^{2}\right)\leq C|\pi|.
\end{eqnarray*}
Then by Burkölder-Davis-Gundy inequality we have
\begin{align}
\E\left[\max_{1\leq i\leq n}\sup_{t_{i-1}\leq t\leq t_i}|Y_{t}-Y_{t_{i-1}}|^{2}\right]\leq C|\pi|.\label{EC}
\end{align}

The estimate for the second term of the left hand in $(\ref{a"61})$ is little involved. First we assume that $b,\sigma,\,\phi^{\pi},\ f,\ g\in C^{1}_{b}$ such that $\phi^{\pi}$ satisfied $(\ref{bound})$.
Let recall $(Y^{\pi},Z^{\pi})$ denote the adapted solution to the BDSDE $(\ref{eqpart})$ and $X^{\pi}$ the solution of the Euler scheme associated to EDS $(\ref{a_{0}})$. Under the Lipschitz conditions on $b$ and $\sigma$, we have
\begin{eqnarray}
\lim_{\pi\rightarrow 0}\max_{1\leq i\leq n} \E\left[\sup_{0\leq t\leq T}|X^{\pi}_{t}-X_{t}|^{2}+\sup_{t_{i-1}\leq t\leq t_{i}}|X_{t}-X_{t_{i-1}}|^{2}\right]=0. \label{a4.}
\end{eqnarray}
Now by the Lipschitz assumption on $\phi^{\pi}$ and $(\ref{a4.})$, applying Lemma 2.2 we know that
\begin{eqnarray}
\lim_{\pi\rightarrow 0}\E\left\{\sup_{0\leq t\leq T}|Y^{\pi}_{t}-Y_{t}|^{2}+\int^{T}_{0}|Z^{\pi}_{t}-Z_{t}|^{2}dt\right\}=0.\label{b1}
\end{eqnarray}
Recalling $(\ref{a"61})$ and applying Lemma 3.4.2 of Zhang \cite{Z}  we have
\begin{eqnarray}
&&\sum_{i=1}^{n}\E\left[\int^{t_i}_{t_{i-1}}|Z_{s}-\tilde{Z}^{\pi}_{t_{i-1}}|^{2}ds\right]\leq \sum_{i=1}
^{n}\E\left[\int^{t_i}_{t_{i-1}}|Z_{s}-Z^{\pi}_{t_{i-1}}|^{2}ds\right]\nonumber\\
&\leq&
2\sum_{i=1}^{n}\E\left[\int^{t_i}_{t_{i-1}}(|Z_{s}-Z^{\pi}_{s}|^{2}+|Z^{\pi}_{s}-Z^{\pi}_{t_{i-1}}|^{2})ds\right].
\label{b'2}
\end{eqnarray}
By $(\ref{b1})$ and $(\ref{b'2})$, to estimate the second term and prove the theorem it remain to show that
\begin{eqnarray}
\sum_{i=1}^{n}\E\left[\int^{t_i}_{t_{i-1}}|Z^{\pi}_{s}-Z^{\pi}_{t_{i-1}}|^{2}ds\right]\leq C|\pi|,\label{R}
\end{eqnarray}
where $C$ is independent of $\pi$.

To do this, let us recall that from Proposition 2.3 of \cite{PP1} and its proof, we know that the martingale part $Z^{\pi}$ has a continuous version given by
\begin{eqnarray*}
Z^{\pi}_{t}=\nabla^{i}Y_{t}^{\pi}[\nabla X^{\pi}_{t}]^{-1}\sigma(X^{\pi}_{t}),\;\; \forall\, t\in[t_{i-1},t_i),\label{Z1}
\end{eqnarray*}
which together with Lemma 2.5 provide
\begin{eqnarray*}
Z^{\pi}_{t}=\left[\left(\xi^{0}_{t}+\sum_{j\geq i}\xi^{j}_{t}\right)M^{-1}_{t}-\int_{0}^{t}f_{x}(\Theta^{\pi}_{r})\nabla X_{r}^{\pi}\Lambda^{-1}_{r}dr-\int_{0}^{t}g_{x}(\Xi^{\pi}_{r})\nabla X_{r}^{\pi}\Lambda^{-1}_{r}d\overleftarrow{B}_r\right]\Lambda_{t}[\nabla
 X^{\pi}_{t}]^{-1}\sigma(X_{t}^{\pi}).
\end{eqnarray*}
Therefore,
\begin{eqnarray}
|Z^{\pi}_{t}-Z^{\pi}_{t_{i-1}}|\leq I_{t}^{1}+I_{t}^{2}+I_{t}^{3}+I_{t}^{4}\label{Z}
\end{eqnarray}
 where
\begin{eqnarray*}
I_{t}^{1}&=&\left|[\xi_{t}^{0}+\sum_{j\geq i}\xi_{t}^{j}]-[\xi_{t_{i-1}}^{0}+\sum_{j\geq t_{i-1}+1}\xi_{t_{i-1}}^{j}]\right|
\times \left|M_{t_{i-1}}^{-1}\Lambda_{t_{i-1}}[\nabla X^{\pi}_{t_{i-1}}]^{-1}\sigma(X^{\pi}_{t_{i-1}}\right|,\\
I_{t}^{2}&=&\left|\xi_{t}^{0}+\sum_{j\geq i}\xi_{t}^{j}\right|\left|M_{t}^{-1}\Lambda_{t}[\nabla X^{\pi}_{t}]^{-1}\sigma(X^{\pi}_{t})-M_{t_{i-1}}^{-1}\Lambda_{t_{i-1}}[\nabla
 X^{\pi}_{t_{i-1}}]^{-1}\sigma(X^{\pi}_{t_{i-1}})\right|,\\
 I_{t}^{3}&=&\left|\int_{0}^{t}f_{x}(r)\nabla X^{\pi}_{r}\Lambda^{-1}_{r}dr\Lambda_{t}[\nabla X^{\pi}_{t}]^{-1}\sigma(X_{t}^{\pi})
-\int_{0}^{t_{i-1}}f_{x}(r)\nabla X^{\pi}_{r}\Lambda^{-1}_{r}dr\Lambda_{t_{i-1}}[\nabla X^{\pi}_{t_{i-1}}]^{-1}\sigma(X_{t_{i-1}}^{\pi})\right|,\\
I_{t}^{4}&=&\left|\int_{0}^{t}g_{x}(r)\nabla X^{\pi}_{r}\Lambda^{-1}_{r}d\overleftarrow{B}_r\Lambda_{t}[\nabla X^{\pi}_{t}]^{-1}\sigma(X_{t}^{\pi})
-\int_{0}^{t_{i-1}}g_{x}(r)\nabla X^{\pi}_{r}\Lambda^{-1}_{r}d\overleftarrow{B}_r\Lambda_{t_{i-1}}[\nabla X^{\pi}_{t_{i-1}}]^{-1}\sigma(X_{t_{i-1}}^{\pi})\right|.
\end{eqnarray*}
Recalling $(\ref{SE})$ and applying Lemma\,$\ref{L1}$ and Lemma\,$\ref{L2}$, one can easily prove that
\begin{eqnarray}
\E(|I_{t}^{3}|^{2}+|I_{t}^{4}|^{2})\leq C|\pi|.\label{I3}
\end{eqnarray}
Recalling $(\ref{xi})$ and $(\ref{bound})$, we have
\begin{eqnarray*}
|\xi^{0}_t+\sum_{j\geq i}\xi^{j}_t|\leq C\E(\sup_{0\leq t\leq T}|\nabla X_t^{\pi}|\mid\mathcal{F}_{t}).
\end{eqnarray*}
Thus by using again Lemma\,$\ref{L1}$ and Lemma\,$\ref{L2}$ one can similarly show that
\begin{eqnarray}
\E(|I_{t}^{2}|^{2})\leq C|\pi|.\label{I2}
\end{eqnarray}
It remains to estimate $I^{1}_t$. To this end we denote
\begin{align*}
\Gamma_t=\sup_{0\leq s\leq t}\left\{1+|X^{\pi}_s|+|[\nabla X^{\pi}_s]^{-1}|+|M^{-1}_s|\right\}.
\end{align*}
Noting that $\Lambda$ is bounded and that $\Gamma_{t_{i-1}}\in\mathcal{F}_{t_{i-1}},$ by $(\ref{xi})$, we have
\begin{eqnarray*}
\E|I_t^{1}|^{2}&\leq &C\E\left\{\Gamma_{t_{i-1}}^{6}\left|[\xi^{0}_{t}+\sum_{j\geq i}\xi^{j}_{t}]-[\xi^{0}_{t_{i-1}}+\sum_{j\geq i}\xi^{j}_{t_{i-1}}]\right|^{2}\right\}\\
&\leq&C\E\left\{\Gamma_{t_{i-1}}^{6}\E\left\{|\xi^{0}_{t}-\xi^{0}_{t_{i-1}}|^{2}+\sum_{j\geq i}|\xi^{j}_{t}-\xi^{j}_{t_{i-1}}|^{2}|\mathcal{F}_{t_{i-1}}\right\}\right\}\\
&\leq &C\E\left\{\Gamma_{t_{i-1}}^{6}\left[\int_{t_{i-1}}^{t_{i}}|\chi^{0}_r|^{2}dr+\int_{t_{i-1}}^{t_{i}}\left|\sum_{j\geq i}\chi^{j}_r\right|^{2}dr\right]\right\}.
\end{eqnarray*}
Therefore, by following the step of \cite{Z}, we get
\begin{eqnarray}
\sum_{i=1}^{n}\E\left(\int_{t_{i-1}}^{t_{i}}|I_t^{1}|^{2}dt\right)\leq C|\pi|\E(\Gamma^{12}_{T})\leq C|\pi|.\label{I1}
\end{eqnarray}
Combining $(\ref{I3}),\, (\ref{I2})$ and $(\ref{I1})$, we infer $(\ref{R})$ from $(\ref{Z})$. This, together with
$(\ref{b'2})$, leads to
\begin{eqnarray*}
\sum_{i=1}^{n}\int^{t_i}_{t_{i-1}}|Z_{s}-\tilde{Z}^{\pi}_{t_{i-1}}|^{2}ds\leq C|\pi|,
\end{eqnarray*}
which ends the estimate of the second term for the smooth case.

In general case, let $b^{\varepsilon},\, \sigma^{\varepsilon},\, \phi^{\pi,\varepsilon},\,f^{\varepsilon}$ and $g^{\varepsilon}$ be molifiers of $b,\,\sigma,\, \phi^{\pi},\, f$ and $g$, respectively, and let $(Y ^{\varepsilon}, Z^{\varepsilon})$ solution of BDSDE
\begin{eqnarray*}
Y_{t}^{\varepsilon,\pi}=\phi^{\pi,\varepsilon}(X_{t_0}^{\varepsilon,\pi},...,X_{t_n}^{\varepsilon,\pi})+\int_{t}^{T}f^{\varepsilon}(\Theta^{\varepsilon,\pi}_{s})\,
ds+\int_{t}^{T}g^{\varepsilon}(\Xi^{\varepsilon,\pi}_{s})d\overleftarrow{B_{s}} -\int_{t}^{T}Z_{s}^{\varepsilon,\pi}dW_{s},\;\; 0\leq t\leq T,
\end{eqnarray*}
where $X^{\varepsilon,\pi}$ is the well-know  Euler approximation of the diffusion $X^{\varepsilon}$, the solution to the corresponding forward SDE $(\ref{a_{0}})$ modified in an obvious way. Then by the above arguments we have
\begin{eqnarray}
\sum_{i=1}^{n}\int^{t_i}_{t_{i-1}}|Z_{s}^{\varepsilon}-\tilde{Z}^{\pi,\varepsilon}_{t_{i-1}}|^{2}ds\leq C|\pi|.\label{G1}
\end{eqnarray}
Therefore using again Lemma 3.4.2 of Zhang, \cite{Z}, we have
\begin{eqnarray}
&&\sum_{i=1}^{n}\int^{t_i}_{t_{i-1}}|Z_{s}-\tilde{Z}^{\pi}_{t_{i-1}}|^{2}ds\leq \sum_{i=1}^{n}\int^{t_i}_{t_{i-1}}|Z_{s}-\tilde{Z}^{\pi,\varepsilon}_{t_{i-1}}|^{2}ds\nonumber\\
&\leq&\sum_{i=1}^{n}\int^{t_i}_{t_{i-1}}[|Z_{s}-Z^{\varepsilon}_{s}|^{2}+|Z_{s}^{\varepsilon}-\tilde{Z}^{\pi,\varepsilon}_{t_{i-1}}|^{2}]ds\nonumber\\
&\leq& C\left\{\E\int^{T}_{0}|Z_{s}-Z^{\varepsilon}_{s}|^{2}ds+|\pi|\right\}.\label{G2}
\end{eqnarray}
Applying Lemma 2.2 we have
\begin{eqnarray*}
\lim_{\varepsilon\rightarrow 0}\E\int^{T}_{0}|Z_{s}-Z^{\varepsilon}_{s}|^{2}ds=0,
\end{eqnarray*}
which, combined with $(\ref{G2})$, proves the estimate of the second term of $(\ref{a"6})$ and together with $(\ref{EC})$ prove the theorem.
\end{proof}

\section{Discrete-time approximation error}
\setcounter{theorem}{0} \setcounter{equation}{0}
In order to approximate the solution of the above decoupled FBDSDE $(\ref{FBSDE})$, we introduce the following discretized version. Let
$\pi:\, t_{0}<t_{1}<.....<t_{n}=T$ be the partition of the time interval $[0,T]$ with mesh
\[|\pi|=\max_{1\leq i\leq n}|t_{i}-t_{i-1}|\] defined in the previous section.
Throughout the rest of the paper, we will use the notations.
\begin{eqnarray*}
\bigtriangleup^{\pi}_{i}=t_{i}-t_{i-1},\,\,\,
\bigtriangleup^{\pi}W_{i}=W_{t_{i}}-W_{t_{i-1}} ,\,\,\mbox{and}\,\,\bigtriangleup^{\pi}B_{i}=B_{t_{i}}-B_{t_{i-1}}\;\; \mbox{for}\,\,
i=1,...,n.
\end{eqnarray*}
The forward component will be approximated by the classical Euler scheme
\begin{eqnarray}
X^{\pi}_{t_{0}}&=& X_{t_{0}}, \nonumber\\\label{a3}\\\nonumber
X^{\pi}_{t_{i}}&=&X^{\pi}_{t_{i-1}}+b(X^{\pi}_{t_{i-1}})\bigtriangleup^{\pi}_{i}
+\sigma(X^{\pi}_{t_{i-1}})\bigtriangleup^{\pi}W_{i}\;\; \mbox{for}\;\; i=1,...,n
\end{eqnarray}
and we set
\begin{eqnarray*}
X^{\pi}_{t}=X^{\pi}_{t_{i-1}}+b(X^{\pi}_{t_{i-1}})(t-t_{i-1})
+\sigma(X^{\pi}_{t_{i-1}})(W_{t}-W_{t_{i-1}})\;\; \mbox{for}\;\; t\in(t_{i-1},t_i).
\end{eqnarray*}
We shall denote by $\displaystyle{\{\mathcal{F}^{\pi}_{t_{i}}\}_{0\leq i\leq n}}$ the
associated discrete-time filtration define by
$$\mathcal{F}^{\pi}_{t_{i}}=\mathcal{F}_{t_{i}}^{W}\vee \mathcal{F}^{B}_{T}.$$

Under the Lipschitz conditions on $b$ and $\sigma$, the following $L{^p}$ estimate for the error due to the Euler scheme is well known
\begin{eqnarray}
\limsup_{\pi\longrightarrow 0}|\pi|^{-1/2}\max_{1\leq i\leq n} \E\left[\sup_{0\leq t\leq T}|X^{\pi}_{t}-X_{t}|^{p}+\sup_{t_{i-1}\leq t\leq t_{i}}|X_{t}-X_{t_{i-1}}|^{p}\right]^{1/p}<\infty, \label{a4}
\end{eqnarray}
for all $p\geq 1$ (see e.g Kloeden and Platen, \cite{KP}). We next consider the following
natural discrete-time approximation of the backward component $Y$:
\begin{eqnarray}
&&Y_{t_n}^{\pi}= h(X^{\pi}_{T}),\; Z_{t_n}^{\pi}=0\nonumber\\\nonumber\\
&&Z^{\pi}_{t_{i-1}}=\frac{1}{\bigtriangleup_{i}^{\pi}}\E^{\pi}_{i-1}[\left(Y_{t_{i}}^{\pi}
+g(t_i,X_{t_i}^{\pi},Y_{t_{i}}^{\pi})\bigtriangleup^{\pi}
B_i\right)\bigtriangleup^{\pi}W_{i}],\label{a5}\\\nonumber\\
&&Y^{\pi}_{t_{i-1}}=\E_{i-1}^{\pi}[Y_{t_{i}}^{\pi}+g(t_{i},X^{\pi}_{t_{i}},Y^{\pi}_{t_{i}})\bigtriangleup^{\pi}B_{i}]
+f(t_{i-1},X^{\pi}_{t_{i-1}},Y^{\pi}_{t_{i-1}},Z^{\pi}_{t_{i-1}})\bigtriangleup_{i}^{\pi},\label{a5'}
\end{eqnarray}
where $\displaystyle{\E_{i}^{\pi}[.]=\E[.|\mathcal{F}_{t_i}^{\pi}]}$.
The above conditional expectation are
well defined at each step of the algorithm. Indeed using the backward induction argument, it easily checked
 that $\displaystyle{Y^{\pi}_{t_{i}}\in L^{2}}$ for all $i$.
\begin{remark}
Using the induction argument, it easily seen that the random variable $Y^{\pi}_{t_{i}}$ and $Z^{\pi}_{t_{i}}$ are $\omega_1$ deterministic function of $X^{\pi}_{t_{i}}$  for each $i=0,...,n$. Then using the fixed point of Banach argument $(\ref{a5'})$ have a unique solution when the mesh of the partition $|\pi|$ is small enough.
\end{remark}

For later use, we need a continuous-time approximation of $(Y,Z)$. Since $Y_{t_{i}}^{\pi}$\newline $+g(t_i,X_{t_i}^{\pi},Y_{t_{i}}^{\pi})\bigtriangleup^{\pi}B_i=\tilde{Y}^{\pi}_{t_i}$ being in $L^{2}$ for all $1\leq i\leq n$, an obvious extension of It\^{o} martingale representation theorem yields the existence of the $\mathcal{F}_{t}$-progressively measurable and square integrable process $Z^{\pi}$ satisfying
\begin{eqnarray}
\tilde{Y}^{\pi}_{t_{i}}=\E[\tilde{Y}^{\pi}_{t_{i}}|\mathcal{F}^{\pi}_{t_{i-1}}]+\int_{t_{i-1}}^{t_{i}}Z^{\pi}_{s}dW_{s}.\label{TR}
\end{eqnarray}
Then we define inductively
\begin{eqnarray}
Y^{\pi}_{t}&=&Y_{t_{i-1}}^{\pi}-(t-t_{i-1})f(t_{i-1},X^{\pi}_{t_{i-1}},Y^{\pi}_{t_{i-1}},Z^{\pi}_{t_{i-1}})
-g(t_i,X^{\pi}_{t_i},Y^{\pi}_{t_i})(B_{t}-B_{t_{i-1}})\nonumber\\
&&+\int^{t}_{t_{i-1}}Z^{\pi}_s dW_s,\;\;\;\;\;\;  t_{i-1}<t\leq t_i .
\label{a6}
\end{eqnarray}
The following property of the $Z^{\pi}$ is needed for the proof of the main result of this
section.
\begin{lemma}
\label{L4}
For all $1\leq i\leq n$, we have
\begin{eqnarray*}
\bigtriangleup_{i}^{\pi}Z^{\pi}_{t_{i-1}}=\E_{i-1}^{\pi}\left[\int^{t_i}_{t_{i-1}}Z^{\pi}_{s}ds\right].
\end{eqnarray*}
\end{lemma}
\begin{proof}
Since
\begin{eqnarray*}
\bigtriangleup^{\pi}_{i}Z^{\pi}_{t_{i-1}}&=&\frac{1}{\bigtriangleup_{i}^{\pi}}\E^{\pi}_{i-1}[\left(Y_{t_{i}}^{\pi}
+g(t_i,X_{t_i}^{\pi},Y_{t_{i}}^{\pi})\bigtriangleup^{\pi}
B_i\right)\bigtriangleup^{\pi}W_{i}],
\end{eqnarray*}
recalling $(\ref{TR})$, we have
\begin{eqnarray*}
Z^{\pi}_{t_{i-1}}=\frac{1}{\bigtriangleup_{i}^{\pi}}\E^{\pi}_{i-1}\left[\bigtriangleup^{\pi}W_{i}\int^{t_i}_{t_{i-1}}Z^{\pi}_s dW_s\right].
\end{eqnarray*}
The result follows by It\^{o}'s isometry.
\end{proof}

We also need the following estimate, which is a particular case of Lemma 2.4.
\begin{lemma}
\label{L5}
For each $1\leq i\leq n$, we define
\begin{eqnarray*}
\tilde{Z}^{\pi}_{t_{i-1}}&=&\frac{1}{\bigtriangleup_{i}^{\pi}}\E_{i-1}^{\pi}\left[\int^{t_i}_{t_{i-1}}Z_{s}ds\right].
\label{a'6}
\end{eqnarray*}
Then
\begin{eqnarray}
&&\limsup_{\pi\rightarrow 0}|\pi|^{-1}\E\left[\max_{1\leq i\leq n}\sup_{t_{i-1}\leq t\leq t_i}|Y_{t}-Y_{t_{i-1}}|^{2}+\sum_{i=1}^{n}\int^{t_i}_{t_{i-1}}|Z_{s}-\tilde{Z}^{\pi}_{t_{i-1}}|^{2}ds\right]< \infty.
\label{a"6}
\end{eqnarray}
\end{lemma}

We are now ready to state our main result of this section, which provides the rate of convergence of the
approximation scheme $(\ref{a5})$ and $(\ref{a5'})$ of the same order than Bouchard and Touzi \cite{Tal}.
\begin{theorem}
\label{T1}
\begin{eqnarray*}
Err_{\pi}(Y,Z)=\left\{\sup_{0\leq t\leq T}\E|Y_{t}-Y_{t}^{\pi}|^{2}+\E\left[\int^{T}_{0}|Z_{s}-Z^{\pi}_{s}|^{2}ds\right]\right\}^{1/2}< C|\pi|^{1/2}.
\end{eqnarray*}
\end{theorem}
\begin{proof}
In the following,  $C>0$ will denote the generic constant independent of $i$ and $n$ that may take values from line to line. Let $i\in\{0,...,n-1\}$ be fixed, and set
\begin{eqnarray*}
\delta^{\pi}Y_{t}&=&Y_{t}-Y_{t}^{\pi},\;\; \delta^{\pi}Z_{t}=Z_{t}-Z_{t}^{\pi},\;\;\delta^{\pi}f(t)=f(t,X_{t},Y_{t},Z_{t})-f(t_{i},X_{t_{i}}^{\pi},Y_{t_{i}}^{\pi},Z_{t_{i}}^{\pi})\\
&&\mbox{and}\;\; \delta^{\pi}g(t)=g(t,X_{t},Y_{t},)-g(t_{i+1},X_{t_{i+1}}^{\pi},Y_{t_{i+1}}^{\pi}),
\end{eqnarray*}
for $t\in[t_{i},t_{i+1})$. By It\^{o}'s formula, we compute that
\begin{eqnarray*}
V_t&=&\E|\delta^{\pi}Y_{t}|^{2}+\E\int^{t_{i+1}}_{t}|\delta^{\pi}Z_{s}|^{2}ds-|\delta^{\pi}Y_{t_{i+1}}|^{2}\\
&=&2\E\int^{t_{i+1}}_{t}\langle
\delta^{\pi}Y_{s},\delta^{\pi}f(s)\rangle ds +\int^{t_{i+1}}_{t}|\delta^{\pi}g(s)|^{2}ds
,\;\;\; t_{i}\leq t\leq t_{i+1}.
\end{eqnarray*}

Let $\beta>0$ be a constant to be chosen later. From Lipschitz property of $f,\ g$ and $h$, together with the
inequality $ab\leq \beta a^{2}+b^{2}/\beta$ this provides
\begin{eqnarray}
V_{t}&\leq& \frac{C}{\beta}\int_{t}^{t_{i+1}}\E\left\{|\pi|^{2}+|X_{s}-X^{\pi}_{t_{i}}|^{2}
+|Y_{s}-Y^{\pi}_{t_{i}}|^{2}+|Z_{s}-Z^{\pi}_{t_{i}}|^{2}\right\}ds\nonumber\\
&&+\int_{t}^{t_{i+1}}C\E\left\{|\pi|^{2}+|X_{s}-X^{\pi}_{t_{i+1}}|^{2}
+|Y_{s}-Y^{\pi}_{t_{i+1}}|^{2}\right\}ds\nonumber\\
&&+\beta\int_{t}^{t_{i+1}}\E|\delta^{\pi}Y_{s}|^{2}ds.
\label{T3.2.1}
\end{eqnarray}
Now observe that
\begin{eqnarray}
\begin{array}{l}
\E|X_{s}-X^{\pi}_{t_{i}}|^{2}+\E|X_{s}-X^{\pi}_{t_{i+1}}|^{2}\leq C|\pi|,\\\\
\E|Y_{s}-Y^{\pi}_{t_{i}}|^{2}\leq 2\left(\E|Y_{s}-Y_{t_{i}}|^{2}
+\E|\delta^{\pi}Y_{t_{i}}|^{2}\right)\leq C\left(|\pi|+\E|\delta^{\pi}Y_{t_{i}}|^{2}\right)\\\\
\E|Y_{s}-Y^{\pi}_{t_{i+1}}|^{2}\leq 2\left(\E|Y_{s}-Y_{t_{i+1}}|^{2}+\E|\delta^{\pi}Y_{t_{i+1}}|^{2}\right)\leq C\left(|\pi|+\E|\delta^{\pi}Y_{t_{i+1}}|^{2}\right)
\end{array}
\label{T3.2.3}
\end{eqnarray}
by $(\ref{a4})$ and $(\ref{a"6})$. Also, with the notation of Lemma $3.3$, it follows from Lemma $3.2$ that
\begin{eqnarray}
\E|Z_{s}-Z^{\pi}_{t_{i}}|^{2}&\leq& 2\left(\E|Z_{s}-\tilde{Z}^{\pi}_{t_{i}}|^{2}+\E|\tilde{Z}^{\pi}_{t_{i}}-Z^{\pi}_{t_{i}}
|^{2}\right)\nonumber\\
& = & 2\left(\E|Z_{s}-\tilde{Z}^{\pi}_{t_{i}}|^{2}+\E\left|\frac{1}{\Delta_{i+1}^{\pi}}\int_{t_{i}}^{t_{i+1}}\E\left(\delta^{\pi}Z_{r}
|\mathcal{F}_{t_{i}}\right)dr\right|^{2}\right)\nonumber\\
&\leq &2\left(\E|Z_{s}-\tilde{Z}^{\pi}_{t_{i}}|^{2}+\frac{1}{\Delta_{i+1}^{\pi}}\int_{t_{i}}^{t_{i+1}}
\E|\delta^{\pi}Z_{r}|^{2}dr\right)
\label{T3.2.4}
\end{eqnarray}
by Jensen's inequality.

We now plug $(\ref{T3.2.3})$ and $(\ref{T3.2.4})$ into $(\ref{T3.2.1})$ to obtain
\begin{eqnarray*}
V_{t}&\leq& \frac{C}{\beta}\int_{t}^{t_{i+1}}\E\left\{|\pi|
+|\delta^{\pi}Y_{t_i}|^{2}+|Z_{s}-\tilde{Z}^{\pi}_{t_{i}}|^{2}\right\}ds\nonumber\\
&&+C\int_{t}^{t_{i+1}}\E\left\{|\pi|
+|\delta^{\pi}Y_{t_{i+1}}|^{2}\right\}ds\nonumber\\
&&+\frac{1}{\Delta^{\pi}_{i+1}}\frac{C}{\beta}\int_{t}^{t_{i+1}}\int_{t_{i}}^{t_{i+1}}\E|\delta^{\pi}Z_{r}|^{2}drds\nonumber\\
&&+\beta\int_{t}^{t_{i+1}}\E|\delta^{\pi}Y_{s}|^{2}ds\label{T3.2.2}\\
&\leq& \frac{C}{\beta}\int_{t}^{t_{i+1}}\E\left\{|\pi|
+|\delta^{\pi}Y_{t_i}|^{2}+|Z_{s}-\tilde{Z}^{\pi}_{t_{i}}|^{2}\right\}ds\nonumber\\
&&+C\int_{t}^{t_{i+1}}\E\left\{|\pi|
+|\delta^{\pi}Y_{t_{i+1}}|^{2}\right\}ds\nonumber\\
&&+\frac{C}{\beta}\int_{t}^{t_{i+1}}\E|\delta^{\pi}Z_{s}|^{2}ds
+\beta\int_{t}^{t_{i+1}}\E|\delta^{\pi}Y_{s}|^{2}ds.
\label{T3.2.5}
\end{eqnarray*}
From the definition of $V_{t}$ and $(\ref{T3.2.5})$, we see that, for $t_{i}\leq t\leq t_{i+1},$
\begin{eqnarray}
\E|\delta^{\pi}Y_{t}|^{2}+\int_{t}^{t_{i+1}}\E|\delta^{\pi}Z_{s}|^{2}ds\leq\beta\int_{t}^{t_{i+1}}\E|\delta^{\pi}Y_{s}|^{2}ds+A_{i}
\label{T3.2.6}
\end{eqnarray}
where
\begin{eqnarray*}
A_{i}&=&(1+C\pi)\E|\delta^{\pi}Y_{t_{i+1}}|^{2}+\frac{C}{\beta}\left[|\pi|^{2}+|\pi|\E|Y^{\pi}_{t_i}|+\int_{t_{i}}^{t_{i+1}}\E|Z_{s}-\tilde{Z}^{\pi}_{t_{i}}|^{2}ds\right]\\
&&+\frac{C}{\beta}\int_{t_{i}}^{t_{i+1}}\E|\delta^{\pi}Z_{s}|^{2}ds.
\end{eqnarray*}
By Gronwall's Lemma, this shows that $\E|\delta^{\pi}Y_{t}|^{2}\leq A_{i}\mbox{e}^{\beta|\pi|}$ for $t_{i}\leq t<t_{i+1},$
which plugged in the second inequality of $(\ref{T3.2.6})$ provides
\begin{eqnarray}
\E|\delta^{\pi}Y_{t}|^{2}+\int_{t}^{t_{i+1}}\E|\delta^{\pi}Z_{s}|^{2}ds &\leq&  A_{i}\left(1+|\pi|\beta\,\mbox{e}^{\beta|\pi|}\right)\leq A_{i}\left(1+C\beta|\pi|\right)
\label{gronw}
\end{eqnarray}
for $|\pi|$ small enough. For $t=t_{i}$ and
$\beta$ sufficiently large than $C$, such that $\frac{C}{\beta}<1$, we deduce from the last inequality
that
\begin{eqnarray*}
&&\E|\delta^{\pi}Y_{t_i}|^{2}+(1-\frac{C}{\beta})\int_{t_i}^{t_{i+1}}\E|\delta^{\pi}Z_{s}|^{2}ds\nonumber\\
&\leq& (1+C|\pi|)\left\{\E|\delta^{\pi}Y_{t_{i+1}}|^{2}+|\pi|^{2}+
\int_{t_{i}}^{t_{i+1}}\E[|Z_{s}-\tilde{Z}^{\pi}_{t_{i}}|^{2}]ds\right\}
\end{eqnarray*}
for small $|\pi|$.

Iterating the last inequality, we get
\begin{eqnarray*}
&&\E|\delta^{\pi}Y_{t_i}|^{2}+(1-\frac{C}{\beta})\int_{t_i}^{t_{i+1}}\E|\delta^{\pi}Z_{s}|^{2}ds\\
&\leq& (1+C|\pi|)^{T/|\pi|}\left\{\E|\delta^{\pi}Y_{T}|^{2}+|\pi|
+\sum_{i=1}^{n}\int_{t_{i-1}}^{t_{i}}\E[|Z_{s}-\tilde{Z}^{\pi}_{t_{i-1}}|^{2}]ds\right\}.
\end{eqnarray*}
Using the estimate $(\ref{a"6})$, together with the Lipschitz property of $g$ and $(\ref{a4})$,
this provides
\begin{eqnarray}
&&\E|\delta^{\pi}Y_{t_i}|^{2}+(1-\frac{C}{\beta})\int_{t_i}^{t_{i+1}}\E|\delta^{\pi}Z_{s}|^{2}ds\nonumber\\
&\leq& (1+C|\pi|)^{T/\pi}\left\{\E|\delta^{\pi}Y_{T}|^{2}+|\pi|+C|\pi|\right\}\leq C|\pi| \label{est1}
\end{eqnarray}
for small $|\pi|$. Summing up inequality $(\ref{gronw})$ with $t=t_i$, we get
\begin{eqnarray*}
&&\left[1-\frac{C}{\beta}(1+C\beta|\pi|)\right]\int_0^T\E|\delta^{\pi} Z_s|^2 ds\\
&\leq& (1+C\beta|\pi|)\frac{C}{\beta}|\pi|+(1+C\beta|\pi|)(1+C|\pi|)\E|\delta^{\pi}Y_T|^2\\
&&+\left[(1+C\beta|\pi|)\frac{C}{\beta}|\pi|-1\right]\E|\delta^{\pi}Y_0|^2\\
&&+\left[(1+C\beta|\pi|)((1+C|\pi|)+\frac{C}{\beta}|\pi|)-1\right]\sum_{i=1}^{n-1}\E|\delta^{\pi}Y_{t_{i}}|^{2}\\
&&+(1+C\beta|\pi|)\frac{C}{\beta}\sum_{i=0}^{n-1}\int_{t_i}^{t_{i+1}}\E|Z_s-\tilde{Z}_{t_i}^{\pi}|^{2}ds.
\end{eqnarray*}
For $\beta$ sufficiently larger that $C$, this proves that for small $|\pi|$:
\begin{eqnarray*}
\int_0^T\E|\delta^{\pi}Z_s|^2 ds&\leq &C\left[|\pi|+\E|\delta^{\pi}Y_T|^2+|\pi|\sum_{i=1}^{n-1}\E|\delta^{\pi}Y_{t_{i}}|^{2}\right.\\
&&+\left.\sum_{i=0}^{n-1}\E|Z_s-\tilde{Z}_{t_i}^{\pi}|^{2}ds\right],
\end{eqnarray*}
where we recall that $C$ is a generic constant which changes from line to line. We now use $(\ref{est1})$ and $(\ref{a"6})$ to see that
\begin{eqnarray*}
\int_0^T\E|\delta^{\pi} Z_s|^2 ds\leq C|\pi|.
\end{eqnarray*}
Together with Lemma\ $\ref{L5}$ and $(\ref{est1})$, this shows that $A_i\leq C|\pi|,$ and therefore,
\begin{eqnarray*}
\sup_{0\leq t\leq T}|\delta^{\pi}Y_{t}|^{2}\leq C|\pi|.
\end{eqnarray*}
by taking the supremum over $t$ in $(\ref{gronw})$. This end the proof of the theorem.
\end{proof}

{\bf Acknowledgments}\newline The author would like to thank I. Boufoussi for his valuable comments and suggestions and express
his deep gratitude to Y. Ouknine and UCAM Mathematics Department for their friendly hospitality during my stay in Cadi Ayyad University.

\label{lastpage-01}
\end{document}